\documentclass[12pt, twoside, leqno]{article}

\usepackage{amsmath,amsthm}
\usepackage{amssymb}

\usepackage{enumitem}

\usepackage{graphicx}

\usepackage[T1]{fontenc}

\newtheorem{theorem}{Theorem}[section]
\newtheorem{corollary}[theorem]{Corollary}
\newtheorem{lemma}[theorem]{Lemma}
\newtheorem{proposition}[theorem]{Proposition}
\newtheorem{problem}[theorem]{Problem}

\newtheorem{mainthm}[theorem]{Main Theorem}

\theoremstyle{definition}
\newtheorem{definition}[theorem]{Definition}

\newtheorem{example}[theorem]{Example}

\newtheorem{claim}[theorem]{Claim}

\newtheorem{conjecture}[theorem]{Conjecture}

\numberwithin{equation}{section}

\newcommand{\R}{\mathbb R}

\newcommand{\F}{\mathcal F}

\newcommand{\w}{\omega}

\newcommand{\Ra}{\Rightarrow}
\newcommand{\var}{\varnothing}

\newcommand{\cl}{\operatorname{cl}}
\newcommand{\cls}{\operatorname{cls}}

\begin{document}


\baselineskip=17pt


\title{Countably compact inverse semigroups and Nyikos' problem}

\author{Serhii Bardyla\\
Institute of Mathematics\\ 
University of Vienna\\
Kolingasse 14-16\\
1090 Vienna, Austria\\
E-mail: sbardyla@gmail.com
}

\date{}

\maketitle


\renewcommand{\thefootnote}{}



\renewcommand{\thefootnote}{\arabic{footnote}}
\setcounter{footnote}{0}


\begin{abstract}
A regular separable first-countable countably compact space is called a {\em Nyikos} space.
In this paper, we give a partial solution to an old problem of Nyikos by showing that each locally compact Nyikos inverse topological semigroup is compact. Also, we show that a topological semigroup $S$ that contains a dense inverse subsemigroup is a topological inverse semigroup, provided (i) $S$ is compact, or (ii) $S$ is countably compact and sequential. The latter result solves a problem of Banakh and Pastukhova and provides the automatic continuity of inversion in certain compact-like inverse semigroups.
\end{abstract}

\section{Introduction}
All topological spaces in this paper are assumed to be {\em Hausdorff}. A regular separable first-countable countably compact space is called a {\em Nyikos} space.
Recall that a space $X$ is called {\em separable} if $X$ has a  countable dense subset; 
{\em first-countable} if each point of $X$ has a countable neighborhood base; {\em countably compact} if each infinite subset of $X$ has an accumulation point.
The following problem posed by Nyikos in 1986 is listed among 20 central problems in Set-theoretic Topology by Hru\v{s}ak and Moore~\cite{HM} (for other occurrences of this problem see~\cite{Ny1,Ny2,Prob}).      


\begin{problem}[Nyikos]\label{pN}
Does ZFC imply the existence of a noncompact Nyikos space?
\end{problem}

Franklin and Rajagopalan~\cite{FR} constructed a normal locally compact noncompact Nyikos space under the consistent assumption $\w_1=\mathfrak t$ (for basic information about cardinal characteristics of the continuum we refer the reader to~\cite{Blass}). Ostaszewski~\cite{Ost} constructed a perfectly normal locally compact hereditary separable noncompact Nyikos space assuming $\diamondsuit$ (for more about the diamond axiom see~\cite[Chapter II.7]{Kun}). Later van Douwen observed that Ostaszewski's arguments yield the existence of a locally compact noncompact Nyikos space under $\mathfrak b=\mathfrak c$. Bardyla and Zdomskyy~\cite{BZ} constructed a consistent example of a Nyikos space which is $\R$-rigid, i.e. admits only constant continuous real-valued functions. Note that $\R$-rigid spaces, being not Tychonoff, are not locally compact. Moreover, Bardyla, Nyikos and Zdomskyy~\cite{BNZ} showed that assuming $\w_1<\mathfrak b=\mathfrak s=\mathfrak c$, every regular separable first-countable non-normal space of weight $<\mathfrak c$ embeds into an $\R$-rigid
Nyikos space.  On the other hand, Weiss~\cite{Weiss} proved that under Martin's Axiom every perfectly normal Nyikos space is compact. Nyikos and Zdomskyy~\cite{NZ} showed that the Proper Forcing Axiom, or briefly PFA, implies that each normal Nyikos space is compact. More information about PFA can be found in~\cite{Moore,Viale}. For other fruitful applications of PFA in Topology see~\cite{PFA2,PFA3,PFA}.

A semigroup $S$ endowed with a topology is called a {\em topological semigroup} if the semigroup operation viewed as a map from $S{\times}S$ to $S$ is continuous.
For the sake of brevity, we call a topological semigroup {\em Nyikos} if its underlying space is Nyikos.  Each topological space $X$ possesses a continuous semigroup operation. For example, for a fixed $a\in X$ one can put $xy=a$ for all $x,y\in X$. Thus the existence of noncompact Nyikos topological semigroups is consistent with ZFC. In this paper, we consider the following natural problem:


\begin{problem}\label{my}
Which Nyikos topological semigroups are compact?    
\end{problem}

Note that any ZFC solution of Problem~\ref{my} can be viewed as a partial answer to Nyikos problem. 
Countably compact topological semigroups and groups are widely studied in Topological Algebra. In particular, much research has been done on convergent sequences in countably compact topological groups~\cite{BRT, BCST, DT,D,GTW,HvM, HMRS, KTW, T1,T2}, and countably compact cancellative semigroups~\cite{Grant,MT,RS,Wall}. Recall that a semigroup $S$ is called {\em cancellative} if for every $a,b,c\in S$ each of the equalities $ab=ac$ and $ba=ca$ implies $b=c$. It is well known that a metrizable countably compact space is compact. By the classical Birkhoff-Kakutani theorem, each first-countable topological group is metrizable.
Then Corollary~5 from~\cite{MT} implies the following.



\begin{theorem}[Mukherjea, Tserpes]\label{MU}
Each first-countable countably compact cancellative topological semigroup is compact.   \end{theorem}



A commutative semigroup of idempotents is called a {\em semilattice}. Groups and semilattices are included in a much larger class of inverse semigroups. Recall that a semigroup $S$ is called {\em inverse} if for each $x\in S$ there exists a unique $y\in X$ such that $xyx=x$ and $yxy=y$. The element $y$ is called the {\em inverse} of $x$ and is denoted by $x^{-1}$. The map $x\mapsto x^{-1}$ is called {\em inversion}. An inverse topological semigroup with continuous inversion is called a {\em topological inverse} semigroup.
The algebraic theory of inverse semigroups is well developed, see the monographs~\cite{Howie,Lawson,Petrich1984}. Topological inverse  semigroups were investigated in~\cite{BB, BBK, BDG, BP, B, our,main,main1, GRep, GR, HS, MPU}.

The following partial solution of Problems~\ref{pN} and~\ref{my} is the main result of this paper.

\begin{mainthm}\label{main}
A locally compact Nyikos inverse topological semigroup is compact.
\end{mainthm}


By $(\w_1,\min)$ and $(\omega_1,\max)$ we denote the first uncountable cardinal $\w_1$ endowed with the semilattice operations $\min$ and $\max$, respectively.
The proof of Main Theorem~\ref{main} is based on the following two results that are of independent interest. The first one characterizes compact Nyikos topological semilattices in terms of their linear subsets (or briefly chains).  

\begin{theorem}\label{semilattice}
For a Tychonoff Nyikos topological semilattice $X$ the following assertions are equivalent:
\begin{enumerate}[label=(\roman*)]
    \item $X$ is compact;
    \item no chain in $X$ is isomorphic to  $(\w_1,\min)$ or $(\w_1,\max)$;
    \item no chain in $X$ is topologically isomorphic to $(\w_1,\min)$ or $(\w_1,\max)$ endowed with the order topology.
\end{enumerate}        
\end{theorem}

The second result solves the problem of Banakh and Pastukhova~\cite[Problem 2.2]{BP} and complements~\cite[Proposition 2.1]{BP},~\cite[Theorem 7]{GR} and~\cite[Corollary 1]{BG}.


\begin{theorem}\label{inverse}
Let $S$ be a topological semigroup that contains a dense inverse subsemigroup. Then $S$ is a topological inverse semigroup provided one of the following conditions holds:
\begin{enumerate}[label=(\roman*)]
   \item $S$ is compact;
  \item  $S$ is countably compact and sequential.
\end{enumerate}   
\end{theorem} 

Recall that a space $X$ is called {\em sequential} if for every non-closed subset $A\subseteq X$ there exists a sequence in $A$ that converges to some point of $X\setminus A$.
Since each first-countable space is sequential, Theorem~\ref{inverse} implies the following.

\begin{corollary}\label{corNyikos}
Let $S$ be a Nyikos topological semigroup that contains a dense inverse subsemigroup. Then $S$ is a topological inverse semigroup.    
\end{corollary}

A space $X$ is called {\em pseudocompact} if $X$ is Tychonoff and each continuous real-valued function on $X$ is bounded.
A semigroup $S$ endowed with a topology is called {\em topologically periodic} if for each $x\in S$ and neighborhood $U$ of $x$ there exists $n\geq 2$ such that $x^n\in U$. 
Theorem~\ref{inverse} is essential for the proof of the following result that complements~\cite[Theorem 2]{BG} and~\cite[Theorem 1]{GR}.

\begin{theorem}\label{autinv}
Let $S$ be an inverse topological semigroup. Then inversion is continuous in $S$ provided one of the following conditions hold:
\begin{enumerate}[label=(\roman*)]
   \item $S$ is Tychonoff and $S{\times}S$ is pseudocompact;
    \item $S$ is regular, topologically periodic and $S{\times}S$ is countably compact.
\end{enumerate}   
\end{theorem}




This paper is organized as follows. Section~\ref{inver} is devoted to compact-like semigroups with dense inverse subsemigroups and the automatic continuity of inversion. In particular, Theorems~\ref{inverse} and~\ref{autinv} are proven there. In Section~\ref{secchain} we investigate chains in Nyikos topological semilattices and prove Theorem~\ref{semilattice}. In Section~\ref{NS} we show that each locally compact Nyikos topological semilattice is compact. This result is a milestone in the proof of Main Theorem~\ref{main} that is given in Section~\ref{proof of the main result}. Also, there we show that Main Theorem~\ref{main} cannot be generalized over simple bands.

\section{Compact-like semigroups with dense inverse subsemigroups and the automatic continuity of inversion}\label{inver}

Let $A$ be a subset of a space $X$. By $\cl_X(A)$ or simply $\overline{A}$ we denote the closure of $A$ in $X$.
Let $$s(A)=\{y\in X\colon \hbox{ }\exists\{x_n: n\in\w\}\subseteq A \hbox{ such that }\lim_{n\in\w}x_n=y\}.$$
The {\em sequential closure} $\cls(A)$ of $A$ is defined recursively as follows. Let $\cls^0(A)=A$ and assume that for each ordinal $\xi<\alpha\leq\omega_1$ the set $\cls^{\xi}(A)$ is defined. Then $\cls^\alpha(A)=\bigcup_{\xi\in\alpha}\cls^{\xi}(A)$ if $\alpha$ is limit and $\cls^{\alpha}(A)=s(\cls^{\gamma}(A))$ if $\alpha=\gamma+1$. Finally, put $\cls(A)=\cls^{\w_1}(A)$. 
Recall that a space $X$ is sequential if and only if $\overline A=\cls(A)$ for every $A\subseteq X$.

\begin{lemma}\label{cls}
Let $A$ be a subsemigroup of a topological semigroup $S$. Then $\cls^{\alpha}(A)$ is a subsemigroup of $S$ for each $\alpha\leq\w_1$.
\end{lemma}

\begin{proof}
Note that $\cls^0(A)=A$ is a subsemigroup of $S$. Assume that for each $\xi<\delta\leq \w_1$, $\cls^\xi(A)$ is a semigroup. If $\delta$ is limit, then $\cls^{\delta}(A)=\bigcup_{\xi<\delta}\cls^\xi(A)$ is a subsemigroup of $S$, being an ascending union of semigroups. Suppose that $\delta=\gamma+1$ for some $\gamma<\w_1$. Then $\cls^\delta(A)=s(\cls^\gamma(A))$. Pick any points $x,y\in \cls^\delta(A)$ and  sequences $\{x_n: n\in\w\}\subseteq \cls^\gamma(A)$ and $\{y_n: n\in\w\}\subseteq \cls^\gamma(A)$ such that $x=\lim_{n\in\w}x_n$ and $y=\lim_{n\in\w}y_n$.
Since $\cls^\gamma(A)$ is a semigroup, we get that the sequence $\{x_ny_n: n\in\w\}$ is contained in $\cls^\gamma(A)$.
Since $S$ is a topological semigroup, the sequence $\{x_ny_n: n\in\w\}$ converges to $xy$. Thus $xy\in \cls^\delta(A)$, witnessing that $\cls^\delta(A)$ is a subsemigroup of $S$. Hence $\cls^{\alpha}(A)$ is a subsemigroup of $S$ for each $\alpha\leq\w_1$, as required.     
\end{proof}

The following problem appears in~\cite[Problem 2.2]{BP}.

\begin{problem}[Banakh, Pastukhova]\label{BP}
Assume that $S$ is a compact topological semigroup that contains a dense inverse subsemigroup. Is $S$ an inverse semigroup?
\end{problem}

For a semigroup $S$ by $E(S)$ we denote the set of all idempotents of $S$. 
A semigroup $S$ is called {\em regular} if for each $x\in S$ there exists $y\in S$ such that $xyx=x$. Note that every inverse semigroup is regular. A regular semigroup $S$ is inverse if an only if $E(S)$ is a semilattice (see~\cite[Theorem 5.1.1]{Howie}). A filter $\F$ on a space $X$ {\em converges} to a point $x\in X$ if $\F$ contains all neighborhoods of $x$. For a subset $F$ of an inverse semigroup $S$ let $F^{-1}=\{x^{-1}:x\in F\}$. 
The following proposition is useful for detecting inverse semigroups among topological semigroups with a dense inverse subsemigroup.

\begin{proposition}\label{tchar}
Let $S$ be a topological semigroup with a dense inverse subsemigroup $X$ and for every $y\in S\setminus X$ there exists an ultrafilter $\F$ convergent to $y$ such that $X\in \F$ and the ultrafilter $\F^{-1}$ generated by the family $\{F^{-1}: X\supseteq F\in \F\}$ converges to some element of $S$. Then $S$ is an inverse semigroup.  
\end{proposition}

\begin{proof}
First let us show that the semigroup $S$ is regular. Fix any $y\in S\setminus X$. By assumption, there exists an ultrafilter $\F$ on $S$ convergent to $y$ such that $X\in\F$ and the ultrafilter $\F^{-1}$ converges to some element $z\in S$. Consider a neighborhood $U$ of the point $w=yzy$. By the continuity of the semigroup operation in $S$, there exist neighborhoods $V(y)$ and $V(z)$ of $y$ and $z$, respectively, such that $V(y)V(z)V(y)\subseteq U$. Then there exist $F_1, F_2\in \F$ such that $F_1\subseteq V(y)\cap X$ and $F_2^{-1}\subseteq V(z)\cap X$. Then for $F=F_1\cap F_2$ we get $F\subseteq FF^{-1}F\subseteq U$. Since the neighborhood $U$ of $w$ is arbitrarily chosen, the filter $\F$ converges to $w$. Since the space $S$ is Hausdorff, any filter on $S$ converges to at most one point, implying that $y=w=yzy$. Hence $S$ is a regular semigroup.  

To prove that $S$ is an inverse semigroup it suffices to show that idempotents of $S$ commute. For this it is enough to show that $E(S)=\overline{E(X)}$, as $E(X)$ is a commutative semigroup. Fix an idempotent $e\in S\setminus X$. By assumption, there exists an ultrafilter $\F$ convergent to $e$ such that $X\in \F$ and the ultrafilter $\F^{-1}$ converges to some point $f\in S$.
Similarly as above one can check that $fef=f$ and $e=efe$.
Since $ee=e$, the continuity of the semigroup operation implies that for every neighborhood $U$ of $e$ there exists a neighborhood $V$ of $e$ such that $VV\subseteq U$. Pick any $F\in\F$  such that $F\subseteq V$ and observe that $FF\subseteq U$. 
Consequently, the filter $\F^2$ generated by the family $\{FF: F\in\F\}$ converges to $e$. 
Fix a neighborhood $U$ of $f$. Since $f=fef$ and $S$ is a topological semigroup, there exist neighborhoods $V(f)$ and $V(e)$ of $f$ and $e$, respectively, such that $V(f)V(e)V(f)\subseteq U$. Since the filter $\F^{-1}$ converges to $f$ and the filter $\F^2$ converges to $e$, there exists $F\in \F$ such that $F\subseteq X$, $F^{-1}\subseteq V(f)$ and $FF\subseteq V(e)$. Fix any $x\in F$. Observe that  $$x^{-1}xxx^{-1}\in F^{-1}FFF^{-1}\subseteq V(f)V(e)V(f)\subseteq U.$$ 
Since $X$ is an inverse semigroup, the elements $x^{-1}x$ and $xx^{-1}$ are idempotents. Since $E(X)$ is a subsemigroup of $X$, the element $x^{-1}xxx^{-1}$ is also an idempotent. As the neighborhood $U$ of $f$ is arbitrarily chosen, we get $f\in \overline{E(X)}$. Since $S$ is a Hausdorff topological semigroup, $f$ is an idempotent. Let $\F^{-2}$ be the filter generated by the family $\{FF: F\in \F^{-1}\}$. Similarly as above one can check that $\F^{-2}$ converges to $f$. Since $S$ is a topological semigroup and $e=efe$, for any neighborhood $W$ of $e$ there exist neighborhoods $V(f)$ and $V(e)$ of $f$ and $e$, respectively, such that $V(e)V(f)V(e)\subseteq W$. As the filter $\F^{-2}$ converges to $f$ and the filter $\F$ converges to $e$, there exists $F\in \F$ such that $F\subseteq X\cap V(e)$ and $(FF)^{-1}=F^{-1}F^{-1}\subseteq V(f)$. Then for any $x\in F$ we have $$E(X)\ni xx^{-1}x^{-1}x\in FF^{-1}F^{-1}F\subseteq V(e)V(f)V(e)\subseteq W.$$
Thus $e\in \overline{E(X)}$ witnessing that $E(S)=\overline{E(X)}$, as required. Hence $S$ is an inverse semigroup.
\end{proof}

A space $X$ is called {\em sequentially compact} if each sequence in $X$ contains a convergent subsequence.
We are in a position to prove Theorem~\ref{inverse} and, in particular, to solve Problem~\ref{BP}. We need to show that a topological semigroup $S$ containing a dense inverse subsemigroup is a topological inverse semigroup provided (i) $S$ is compact, or (ii) $S$ is countably compact and sequential.

\begin{proof}[{\bf Proof of Theorem~\ref{inverse}}]
(i) Let $S$ be a compact topological semigroup containing a dense inverse subsemigroup $X$. 
Fix any $y\in S\setminus X$ and consider an ultrafilter $\F$ on $S$ that contains the family $\{U\cap X: U$ is a neighborhood of $y\}$. It is clear that $\F$ converges to $y$. By the compactness of $S$, the ultrafilter $\F^{-1}$ generated by the family $\{F^{-1}:X\supseteq F\in\F\}$ converges to some point of $S$. Proposition~\ref{tchar} implies that $S$ is an inverse topological semigroup. By~\cite[Proposition 1.6.7]{Hajii}, inversion is automatically continuous in compact inverse topological semigroups. Hence $S$ is a topological inverse semigroup.

(ii) Since the space $S$ is sequential we get that $$S=\overline{X}=\cls(X)=\cls^{\w_1}(X)=\bigcup_{\xi<\w_1}\cls^{\xi}(X).$$ First let us show that $S$ is an inverse semigroup. Note that $\cls^0(X)=X$ is an inverse subsemigroup of $S$.
Assume that $\cls^{\xi}(X)$ is an inverse subsemigroup of $S$ for every $\xi<\alpha\leq \w_1$. If the ordinal $\alpha$ is limit, then 
$\cls^{\alpha}(X)=\bigcup_{\xi<\alpha}\cls^\xi(X)$ is an inverse semigroup, being an ascending union of inverse semigroups. 
Suppose that $\alpha=\gamma+1$ for some $\gamma<\omega_1$. Lemma~\ref{cls} implies that $\cls^{\alpha}(X)=s(\cls^{\gamma}(X))$ is a subsemigroup of $S$ that contains densely the inverse semigroup $\cls^{\gamma}(X)$.
Fix any $x\in \cls^{\alpha}(X)\setminus \cls^{\gamma}(X)$ and a sequence $\{x_n:n\in\w\}\subseteq \cls^{\gamma}(X)$ convergent to $x$. 
By~\cite[Theorem 3.10.31]{Eng}, each Hausdorff countably compact sequential space is sequentially compact. Thus the sequence $\{x_n^{-1}: n\in\w\}\subseteq \cls^{\gamma}(X)$ contains a convergent in $S$ subsequence $\{x^{-1}_{n_k}: k\in\w\}$. Set $y=\lim_{k\in\w}x^{-1}_{n_k}$ and note that $y\in s(\cls^{\gamma}(X))=\cls^{\alpha}(X)$. Fix any free ultrafilter $\F$ on $\cls^{\alpha}(X)$ which contains the set $\{x_{n_k}:k\in\w\}$.  It is easy to see that $\F$ converges to $x$, and the ultrafilter $\F^{-1}$ generated by the family $\{F^{-1}\colon \cls^{\gamma}(X)\supseteq F\in\F\}$ converges to $y$. Proposition~\ref{tchar} implies that $\cls^{\alpha}(X)$ is an inverse semigroup. Hence $S=\cls^{\w_1}(X)$ is an inverse semigroup. 
By~\cite[Corollary 1]{BG}, inversion is automatically continuous on every countably compact sequential inverse topological semigroup.
Hence, $S$ is a topological inverse semigroup.
\end{proof}

The {\em bicyclic monoid} is generated by two elements $p,q$ subject to one condition $qp=1$. It is well known that the bicyclic monoid is an inverse semigroup (see~\cite[Chapter 3.4]{Lawson}). 
The following two results that appear in~\cite[Theorem 6.1 and Theorem 6.6]{BDG} establish some sharpness of Theorem~\ref{inverse}.

\begin{theorem}[Banakh, Dimitrova, Gutik]
There exists a pseudocompact topological semigroup which is not an inverse semigroup but contains a dense copy of the bicyclic monoid.
\end{theorem}

\begin{theorem}[Banakh, Dimitrova, Gutik]
If there exists a torsion-free abelian countably compact topological group without non-trivial convergent sequences, then there exists a Tychonoff countably compact topological semigroup which is not an inverse semigroup but contains a dense copy of the bicyclic monoid.
\end{theorem}

Noteworthy, the existence of a torsion-free abelian countably compact topological group without non-trivial convergent sequences is still not established in ZFC. However, various consistent examples of such groups were constructed in~\cite{BRT, BCST, DT,D,GTW,HvM, KTW, T1,T2}.


Recall that an inverse semigroup $S$ is called {\em Clifford} if $xx^{-1}=x^{-1}x$ for all $x\in S$. The automatic continuity of inversion in groups, Clifford semigroups, and inverse semigroups is widely studied (see~\cite{BG, BR, Brand, Ellis, GR, Pfi, Rez, Roma, Tka}). In most cases the continuity of inversion follows from the continuity of multiplication and some compact-like property. For instance, the following two results appear in~\cite[Corollary 2]{GR} and~\cite[Theorem 2]{BG}, respectively.


\begin{theorem}[Gutik, Pagon, Repov\v{s}]\label{GR}
Let $S$ be a Tychonoff Clifford topological semigroup such that $S{\times}S$ is pseudocompact. Then inversion is continuous in $S$. 
\end{theorem}

\begin{theorem}[Banakh, Gutik]\label{BG1}
Let $S$ be a regular topologically periodic Clifford topological semigroup such that $S{\times}S$ is countably compact. Then inversion is continuous in $S$.    
\end{theorem}

The rest of this section is devoted to the proof of Theorem~\ref{autinv}, which in fact generalizes Theorems~\ref{GR} and~\ref{BG1}. 
A semigroup $S$ endowed with a topology is called {\em semitopological}  if for every $s\in S$ the shifts $l_s: x\mapsto sx$ and $r_s: x\mapsto xs$ are continuous in $S$. It is clear that every topological semigroup is semitopological.


\begin{proposition}\label{periodic}
Each inverse topologically periodic semitopological semigroup $S$ is Clifford.  
\end{proposition}

\begin{proof}
Fix an element $x\in S$ and a neighborhood $U$ of $x$. Since $X$ is topologically periodic, there exists a positive integer $n\geq 2$ such that $x^n\in U$. Then $$x^n=x^2x^{-2}x^n\in U\cap (x^{2}x^{-2})U,$$
and 
$$x^n=x^nx^{-2}x^2\in U(x^{-2}x^{2})\cap U.$$
It follows that $V\cap (x^2x^{-2})V\neq \var$ and  $V(x^{-2}x^{2})\cap V\neq \var$ for each neighborhood $V$ of $x$. Since $X$ is a Hausdorff semitopological semigroup, we obtain $x^{2}x^{-2}x=x=xx^{-2}x^2$. It follows that $x^{-1}x^{2}x^{-2}=x^{-1}=x^{-2}x^2x^{-1}$.
Then
\begin{align*}
x^{-1}x=x^{-1}x^2x^{-2}x^2x^{-2}x&=x^{-1}x^2x^{-2}x=\\
&=(x^{-1}x)(xx^{-1})(x^{-1}x)=(xx^{-1})(x^{-1}x).
\end{align*}
In the formula above, the second equality follows from the fact that $x^{2}x^{-2}$ is an idempotent, whereas the last equality holds as idempotents commute in $S$. 
Similarly
\begin{align*}
xx^{-1}=xx^{-2}x^{2}x^{-2}x^2x^{-1}&=xx^{-2}x^{2}x^{-1}=\\
&=(xx^{-1})(x^{-1}x)(xx^{-1})=(xx^{-1})(x^{-1}x).
\end{align*}
Hence $xx^{-1}=x^{-1}x$ witnessing that the semigroup $S$ is Clifford.
\end{proof}

By $\beta X$ we denote the Stone-\v{C}ech compactification of a Tychonoff space $X$ (see~\cite[Chapter~3.6]{Eng}). 
The following result appears in~\cite[Theorem 2.3]{BDG}.

\begin{theorem}[Banakh, Dimitrova, Gutik]\label{BDG}
Let $S$ be a Tychonoff topological semigroup  such that $S{\times}S$ is pseudocompact. Then the semigroup operation of $S$ extends to a continuous semigroup operation on $\beta S$.    
\end{theorem}

We are in a position to prove Theorem~\ref{autinv}. We need to show that inversion is continuous in any inverse topological semigroup $S$ provided (i) $S{\times}S$ is pseudocompact, or (ii) $S$ is regular, topologically periodic and $S{\times}S$ is countably compact.

\begin{proof}[{\bf Proof of Theorem~\ref{autinv}}]
(i)    By Theorem~\ref{BDG}, $\beta S$ is a compact topological semigroup that contains $S$ as a dense inverse subsemigroup. Theorem~\ref{inverse} implies that $\beta S$ is a topological inverse semigroup. Then the continuity of inversion in $\beta S$ yields the continuity of inversion in $S$.  

(ii) The proof follows from Proposition~\ref{periodic} and Theorem~\ref{BG1}.    
\end{proof}

\section{Chains in topological semilattices}\label{secchain}
Recall that each semilattice possesses the natural partial order $\leq$ defined by $a\leq b$ if and only if $ab=ba=a$. Further, if a semilattice is treated as a poset, then by default it is assumed to carry the natural partial order.
For an element $x$ of a semilattice $X$ put ${\downarrow}x=\{y\in X\colon y\leq x\}$ and ${\uparrow}x=\{y\in X\colon x\leq y\}$. Note that if $X$ is a semitopological semilattice, then for every $x\in X$ the sets ${\downarrow}x$ and ${\uparrow}x$ are closed.
A semitopological semilattice $X$ is called {\em chain-compact} if all maximal chains in $X$ are compact. The following two theorems follow from~\cite[Theorem 3.1]{BB}. 

\begin{theorem}[Banakh, Bardyla]\label{BBchain}
For a semitopological semilattice $X$ the following is equivalent:
\begin{enumerate}[label=(\arabic*)] 
    \item $X$ is chain-compact;
    \item  each nonempty chain $L\subseteq X$ has $\inf L\in \overline{L}$ and $\sup L\in \overline{L}$;
    \item no closed chain in $X$ is topologically isomorphic to an infinite regular cardinal endowed with the order topology and the semilattice operation of minimum or maximum.
\end{enumerate}  
\end{theorem}

\begin{theorem}[Banakh, Bardyla]\label{H-closed}
If $X$ is a chain-compact subsemilattice of a topological semigroup $Y$, then $X$ is closed in $Y$.  
\end{theorem}

The following fact is well known.

\begin{proposition}[Folklore]\label{chainfolk}
The closure of a chain in a semitopological semilattice is a chain.    
\end{proposition}

\begin{lemma}\label{conv}
Let $X$ be a semitopological semilattice and $L\subseteq X$ be a chain isomorphic to $(\w_1,\min)$. If $\overline{L}$ is compact, then there exists $z=\sup L\in \overline{L}$ such that for each neighborhood $U$ of $z$ the set $L\setminus U$ is countable.   
\end{lemma}

\begin{proof}
By Proposition~\ref{chainfolk}, $\overline{L}$ is a compact chain. Theorem~\ref{BBchain}(2) applied to $\overline{L}$ implies that $\overline{L}$ contains the supremum of $L$ which we denote by $z$. It is clear that $z=\sup L'$ for each uncountable subset $L'$ of $L$. By Theorem~\ref{BBchain}(2), $z\in\overline{L'}$ for each uncountable subset $L'$ of $L$.  It follows that for each neighborhood $U$ of $z$ the set $L\setminus U$ is countable, as required.
\end{proof}

Dually, one can prove the following.

\begin{lemma}\label{conv1}
Let $X$ be a semitopological semilattice and $L\subseteq X$ be a chain isomorphic to $(\w_1,\max)$. If $\overline{L}$ is compact, then there exists $z=\inf L\in \overline{L}$ such that for each neighborhood $U$ of $z$ the set $L\setminus U$ is countable.   
\end{lemma}

Recall that a space $X$ is called {\em countably tight} if for every $A\subseteq X$ and $x\in\overline A$ there exists a countable subset $B\subseteq A$ such that $x\in\overline B$. 

\begin{lemma}\label{base}
Let $X$ be a countably tight semitopological semilattice and $L$ be a chain isomorphic to $(\w_1,\min)$ or $(\w_1,\max)$. Then $\overline{L}$ is not compact.
\end{lemma}

\begin{proof}
Let $L$ be a chain isomorphic to $(\w_1,\min)$. If $\overline{L}$ is compact, then Lemma~\ref{conv} yields $z=\sup L\in\overline{L}$. The countable tightness of $X$ implies that there exists a countable subset $A$ of $L$ such that $z\in\overline{A}$. Then there exists $l\in L$ such that $A\subseteq {\downarrow} l$. Since $X$ is a semitopological semilattice, $X\setminus {\downarrow}l$ is a neighborhood of $z$ disjoint with $A$, a contradiction. 
In case when $L$ is isomorphic to $(\w_1,\max)$, the proof is similar.
\end{proof}

We are in a position to prove Theorem~\ref{semilattice}. We need to show that for a Tychonoff Nyikos topological semilattice $X$ the following conditions are equivalent: (i) $X$ is compact; (ii) no chain in $X$ is isomorphic to $(\w_1,\min)$ or $(\w_1,\max)$; (iii) no chain in $X$ is topologically isomorphic to $(\w_1,\min)$ or $(\w_1,\max)$ equipped with the order topology. 

\begin{proof}[{\bf Proof of Theorem~\ref{semilattice}}]
Since each first-countable space is countably tight, the implication (i) $\Rightarrow$ (ii) follows from Lemma~\ref{base}. The implication (ii) $\Rightarrow$ (iii) is obvious.

(iii) $\Rightarrow$ (i): Since $X$ is first-countable, for any cardinal $\kappa>\w_1$ there is no chain in $X$ that is topologically isomorphic to $(\kappa,\min)$ or $(\kappa,\max)$ endowed with the order topology. Since $X$ is countably compact, no closed chain in $X$ is topologically isomorphic to the discrete semilattices $(\w,\min)$ and $(\w,\max)$. By assumption, no chain in $X$ is topologically isomorphic to $(\w_1,\min)$ or $(\w_1,\max)$ endowed with the order topology. Theorem~\ref{BBchain}(3) implies that the semilattice $X$ is chain-compact. By \cite[Corollary 3.10.15]{Eng}, $X{\times}X$ is countable compact and thus pseudocompact. By Theorem~\ref{BDG}, $\beta X$ is a topological semigroup that contains $X$ as a dense subsemilattice. Theorem~\ref{H-closed} implies that $X$ is closed in $\beta X$ and thus $X=\beta X$ is compact.   
\end{proof}

The rest of this section contains technical results on chains in topological semilattices. They will be used to prove the compactness of locally compact Nyikos topological semilattices (see Theorem~\ref{Nyikossemilattice}), which is a milestone in the proof of Main Theorem~\ref{main}.

\begin{lemma}\label{folknew}
Let $\alpha$ be an ordinal and $Y$ be a semilattice. Then the following assertions hold:
\begin{enumerate}[label=(\arabic*)] 
    \item if $h:(\alpha,\min)\rightarrow Y$ is a homomorphism, then $h(\alpha)$ is isomorphic to $(\beta,\min)$ for some $\beta\leq\alpha$;
    \item if $h:(\alpha,\max)\rightarrow Y$ is a homomorphism, then $h(\alpha)$ is isomorphic to $(\beta,\max)$ for some $\beta\leq\alpha$.
\end{enumerate}     
\end{lemma}

\begin{proof}
(1) Consider the subsemilattice $A=\{\min (h^{-1}(x)): x\in h(\alpha)\}$ of $(\alpha,\min)$. Since $A\subseteq \alpha$, there exists an ordinal $\beta\leq \alpha$ such that $A$ is isomorphic to $(\beta,\min)$. It is straightforward to check that $h(\alpha)$ is isomorphic to $A$ and thus to $(\beta,\min)$. 

The proof of (2) is analogous. 
\end{proof}


\begin{lemma}\label{sup}
Let $X$ be a countably compact semitopological semilattice, $x\in X$, and $C$ be a chain isomorphic to $(\w,\min)$. Then the following conditions are equivalent:
\begin{enumerate}[label=(\arabic*)] 
    \item For each neighborhood $U$ of $x$ the set $C\setminus U$ is finite;
    \item $x\in \overline{C}$;
    \item $x=\sup C$.
\end{enumerate} 
\end{lemma}

\begin{proof}

The implication (1) $\Rightarrow$ (2) is obvious. 

(2) $\Ra$ (3): Assume that $x\in \overline C$. Fix any $c\in C$ and neighborhood $U$ of $x$. Note that the set ${\uparrow}c\cap U$ is nonempty, witnessing that $c\in cU$. Since $U$ was arbitrarily chosen and $X$ is a Hausdorff semitopological semilattice, we get that $cx=c$. Since the point $c$ was chosen arbitrarily, we obtain that $c\leq x$ for all $c\in C$, i.e. $x$ is an upper bound of $C$. Assume that $y$ is another upper bound of $C$. Taking into account that $x\in\overline{C}$,  we have $U\cap yU\neq \varnothing$ for each neighborhood $U$ of $x$. Since $X$ is a Hausdorff semitopological semilattice, $x=yx$. Hence $x=\sup C$. 

(3) $\Ra$ (1):  To derive a contradiction, assume that there exists a neighborhood $U$ of $x$ such that the set $D=C\setminus U$ is infinite. It is clear that $D$ is a cofinal subset of $C$ isomorphic to $(\w,\min)$. In particular, we have $\sup D=x$. By the countable compactness of $X$, there exists $y\in \overline D$. Repeating the arguments above, we obtain $y=\sup D$. Hence $y=x$, but $x\notin \overline{D}$, a contradiction.
\end{proof}



    

\begin{lemma}\label{+1}
Let $X$ be a semitopological semilattice and $L\subseteq X$ be a chain isomorphic to $(\alpha,\min)$ for some ordinal $\alpha$. If $\overline{L}$ does not contain maximum, then $L$ is cofinal in $\overline{L}$ and $\overline{L}$ is isomorphic to $(\alpha,\min)$.    
\end{lemma}  

\begin{proof}
Let $L=\{l_\xi:\xi\in\alpha\}$, where $l_\xi\leq l_{\eta}$ if and only if $\xi\leq \eta$. By Proposition~\ref{chainfolk}, $\overline{L}$ is a chain.  Fix any $x\in\overline{L}$. If $L\subseteq {\downarrow}x$, then $\overline{L}\subseteq {\downarrow}x$, as ${\downarrow}x$ is closed in $X$. But then $x$ is the maximum of $\overline{L}$, a contradiction. So, for each $x\in\overline{L}$, the set $L\setminus {\downarrow}x$ is nonempty. It follows that the ordinal $\alpha$ is limit and $L$ is cofinal in $\overline{L}$.
For every $x\in \overline{L}\setminus L$ let $\delta(x)=\min\{\xi\in\alpha: x\leq l_\xi\}$. Since $L$ is cofinal in $\overline{L}$, the ordinals $\delta(x)$, $x\in \overline{L}\setminus L$, are well defined. 

Seeking a contradiction, assume that there exists an infinite decreasing chain $\{x_n:n\in\w\}\subseteq \overline{L}$. The chain $L$, being isomorphic to an ordinal, contains only finite decreasing chains. Hence we lose no generality assuming that $\{x_n:n\in\w\}\subseteq \overline{L}\setminus L$. Observe that $\{\delta(x_n):n\in\w\}$ is a nonincreasing chain in $\alpha$. Thus there exists $k\in\w$ such that $\delta(x_n)=\delta(x_m)$ for each $n,m\geq k$. Then the set $W=X\setminus ({\uparrow}x_k\cup {\downarrow}x_{k+2})$ is a neighborhood of $x_{k+1}$ which is disjoint with $L$. The obtained contradiction implies that $\overline{L}$ contains only finite decreasing chains. Hence $\overline{L}$ is isomorphic to $(\theta,\min)$ for some ordinal $\theta$. Fix the order isomorphism $\psi:\overline{L}\rightarrow \theta$. It is easy to see that $\psi$ can be defined  recursively as follows: $\psi(l_0)=0$ and for each $y\in\overline{L}$,
$$\psi(y)=
\begin{cases}
 \psi(x)+1, & \hbox{if }\exists x<y \hbox{ such that } \{z\in \overline L: x<z<y\}=\varnothing;\\
 \sup\{\psi(x):x< y\}, & \hbox{otherwise}.
\end{cases}
$$

Since $\psi(l_\xi)\geq \xi$ for all $\xi\in\alpha$, we get $\alpha\leq \theta$.

\begin{claim}\label{claimineq}
$\psi(l_{\xi})\leq \xi+1$ for every $\xi<\alpha$.    
\end{claim}

\begin{proof}
Assume that $\psi(l_\delta)<\delta+1$ for all $\delta<\xi$. There are two cases to consider:
\begin{enumerate}[label=(\roman*)]
    \item $\xi=\mu+1$ for some ordinal $\mu<\alpha$;
    \item $\xi$ is a limit ordinal.
\end{enumerate}

In case (i) we have $\psi(l_\mu)\leq \mu+1=\xi$. Observe that the set $A=\overline{L}\setminus({\uparrow}l_\xi\cup{\downarrow}l_{\mu})$ is open in $\overline{L}$ and disjoint with $L$. Therefore $A=\var$. Hence there exists no $z\in\overline{L}$ such that $l_\mu<z<l_\xi$. The definition of $\psi$ implies that $\psi(l_\xi)=\psi(l_\mu)+1\leq \xi+1$.

Consider case (ii). Let $\pi$ be the supremum of the set $\{l_{\delta}:\delta<\xi\}$ in $\overline{L}$. It is clear that $\pi\leq l_\xi$. Since the set $\{l_\delta:\delta<\xi\}$ is cofinal in $\{x\in\overline{L}:x<\pi\}$ and $\psi$ is an order isomorphism, we get that $$\psi(\pi)=\sup\{\psi(l_\delta):\delta<\xi\}\leq \sup\{\delta+1:\delta<\xi\}=\xi,$$
where the inequality above follows from the inductive assumption. So, if $\pi=l_\xi$, then we are done. Assume that $\pi<l_\xi$. Observe that the set $\overline{L}\setminus({\downarrow}\pi\cup {\uparrow}l_\xi)$ is empty, as it is open in $\overline{L}$ and disjoint with $L$. Hence there exists no $z\in\overline{L}$ such that $\pi<z<l_\xi$. Then 
$\psi(l_\xi)=\psi(\pi)+1\leq \xi+1$. 
\end{proof}

Since the ordinal $\alpha$ is limit, $\alpha=\sup\{\xi+1: \xi<\alpha\}$. Observe that $\psi(L)$ is cofinal in $\theta$, as $\psi$ is an order isomoprhism. Then   
Claim~\ref{claimineq} yields $\theta\leq \sup\{\xi+1: \xi<\alpha\}=\alpha$. Thus $\theta=\alpha$, as required. 
\end{proof}





\begin{lemma}\label{noproofi}
Let $X$ be a countably compact countably tight semitopological semilattice and $L\subseteq X$ be a chain isomorphic to $(\w_1,\min)$. Then $\overline{L}$ is topologically isomorphic to $(\w_1,\min)$ endowed with the order topology, and $L$ is cofinal in $\overline{L}$.       
\end{lemma}

\begin{proof}
Fix any $x\in\overline{L}\setminus L$. Since $X$ is countably tight, there exists a countable subset $A$ of $L$ such that $x\in \overline{L}$. It is clear that there exists $l\in L$ such that $A\subseteq {\downarrow}l$. Since $X$ is a semitopological semillatice, the set ${\downarrow}l$ is closed and hence $x\in \overline{A}\subseteq {\downarrow}l$. It follows that for each $x\in \overline{L}$ there exists $l\in L$ with $x\leq l$. Thus, $\overline{L}$ does not contain maximum. Lemma~\ref{+1} implies that $\overline{L}$ is isomorphic to $(\w_1,\min)$ and $L$ is cofinal in $\overline{L}$. 
Let $\overline{L}=\{y_\xi:\xi<\w_1\}$, where $y_\xi\leq y_{\delta}$ if and only if $\xi\leq \delta$. Fix any $\xi<\w_1$.  Clearly, for each $\delta<\xi$ the set $$\{y_\alpha: \delta< \alpha<{\xi+1}\}=\overline{L}\setminus ({\downarrow}y_\delta\cup {\uparrow}y_{\xi+1})$$ is open in $\overline{L}$. Hence the order topology on $\overline{L}$ is contained in the original one. To derive a contradiction, assume that the order topology on $\overline{L}$ is strictly coarser than the original one. Then, taking into account that for each successor ordinal $\theta<\w_1$ the point $y_\theta$ is isolated in both topologies, there exists a limit ordinal $\xi<\w_1$ and a neighborhood $U$ of $y_\xi$ such that for any $\delta<\xi$ the set $\{y_\alpha:\delta<\alpha<\xi+1\}\setminus U$ is not empty. Since $\xi$ is countable, there exists a cofinal in ${\downarrow}y_\xi\setminus \{y_\xi\}$ subset $\{\pi_n:n\in\w\}\subseteq\overline{L}\setminus U$. But this contradicts Lemma~\ref{sup}, as $y_\xi=\sup\{\pi_n:n\in\w\}$, but $y_\xi\notin\overline{\{\pi_n:n\in\w\}}$.
\end{proof}

\section{Nyikos semilattices}\label{NS}
The aim of this section is to show that each locally compact Nyikos topological semilattice is compact (see Theorem~\ref{Nyikossemilattice}).  

\begin{lemma}\label{0}
Each separable countably compact semitopological semilattice possesses the minimum.   
\end{lemma}

\begin{proof}
Let $D$ be a countable dense subset of a countably compact semitopological semilattice $X$. Recall that for each $x\in X$ the set ${\downarrow}x$ is closed. Since $X$ is countably compact and the family $\{{\downarrow}x:x\in X\}$ has the finite intersection property, $Z=\bigcap_{d\in D}{\downarrow}d\neq \var$. Pick any $z\in Z$ and note that $zX=z\overline{D}\subseteq \overline{zD}=\overline{ \{z\}}=\{z\}$. Hence $z$ is the minimum of $X$.   
\end{proof}

\begin{definition}
Let $X$ be a semilattice and $L\subseteq X$ be a chain. Denote $I_L=\{x\in X: |xL|<|L|\}$.  
\end{definition}

\begin{lemma}\label{idealIL}
For every semilattice $X$ and chain $L\subseteq X$ the set $I_L$ is either empty or an ideal.   
\end{lemma}

\begin{proof}
Suppose that $I_L\neq \varnothing$. Fix any $x\in X$ and $y\in I_L$. Since $|yL|<|L|$, we get $|xyL|=|x(yL)|\leq |yL|<|L|$. Hence $xy\in I_L$, witnessing that $I_L$ is an ideal.    
\end{proof}

\begin{lemma}\label{IL}
Let $X$ be a countably tight semitopological semilattice and $L\subseteq X$ be a chain isomorphic to $(\w_1,\min)$. Then $I_L$ is a closed ideal.     
\end{lemma}

\begin{proof}
Observe that $L\subseteq I_L$ and hence $I_L\neq \var$. 
By Lemma~\ref{idealIL}, $I_L$ is an ideal.
Let $L=\{l_\xi:\xi < \w_1\}$, where $l_\xi\leq l_\beta$ if and only if $\xi\leq \beta$. Fix any adherent point $z$ of $I_L$. Since $X$ is countably tight, there exists a countable subset $B\subseteq I_L$ such that $z\in \overline{B}$.
For each $b\in B$ there exists an ordinal $\alpha_b<\w_1$ such that $b l_{\xi}=b l_{\alpha_b}$ for all $\xi\geq \alpha_b$. Let $\alpha=\sup\{\alpha_b:b\in B\}$. Then for each $\xi\geq \alpha$ and $b\in B$ we have $b l_{\xi}=b l_{\alpha}$. Since $z\in \overline{B}$, for each neighborhood $U$ of $z$ and $\xi\geq \alpha$ we have $Ul_\alpha\cap Ul_\xi\neq \varnothing$. Taking into account that $X$ is a Hausdorff semitopological semilattice, we conclude that $zl_\alpha=zl_\xi$ for all $\xi\geq \alpha$. It follows that $|zL|<\w_1$ and hence $z\in I_L$. 
\end{proof}

Similarly one can show the following.

\begin{lemma}\label{IL1}
Let $X$ be a countably tight semitopological semilattice and $L\subseteq X$ be a chain isomorphic to $(\w_1,\max)$ such that $I_L\neq \var$. Then $I_L$ is a closed ideal.    
\end{lemma}



Note that the order topologies on semilattices $(\w_1,\max)$ and $(\w_1,\min)$ coincide. 
The following result follows from~\cite[Theorem 1]{BBK}

\begin{theorem}[Banakh, Bonnet, Kubi\'s]\label{BBK}
Let $X$ be a separable topological semilattice that contains a subspace $L$ homeomorphic to $\w_1$ endowed with the order topology. Then $|X\setminus L|>\w$.    
\end{theorem}

\begin{lemma}\label{good}
Let $X$ be a Nyikos topological semilattice and $L\subseteq X$ be a chain isomorphic to $(\w_1,\min)$. Then there exists $x\in L$ such that $x\in \overline{X\setminus I_L}$.  
\end{lemma}

\begin{proof}
Let $L=\{l_{\alpha}:\alpha<\w_1\}\subseteq X$, where $l_{\alpha}\leq l_{\beta}$ if and only if $\alpha\leq \beta$. It is clear that $L\subseteq I_L$. 
Fix any $x\in I_L$. 
Since $xL$ is countable, there exists $\delta<\w_1$ such that $xl_{\xi}=xl_{\delta}$ for each $\xi\geq \delta$. 
By Lemma~\ref{noproofi}, the chain $\overline{L}$ is topologically isomorphic to $(\w_1,\min)$ equipped with the order topology, and contains $L$ as a cofinal subset. Hence $xl_{\delta}=xa$ for each $a\in \overline L\cap {\uparrow}l_{\delta}$, implying that the set $x\overline L$ is countable. It follows that $I_L=I_{\overline{L}}$. Since $X$ is separable, it contains a countable dense subsemilattice $D$. Seeking a contradiction, assume that there exists an open set $U\supseteq L$ such that $U\subseteq  I_L$.
Consider the subsemilattice $Y$ of $X$ generated by $\overline{L}\cup (D\cap U)$.
Since $U\subseteq I_L=I_{\overline{L}}$ we obtain that the set 
$$D'=Y\setminus \overline{L}\subseteq (D\cap I_L)\cup\bigcup_{x\in D\cap I_L}x\overline{L}$$ 
is countable.  Since $D$ is dense in $X$ and $U\supseteq L$ is open, we get that $L\subseteq \overline{U\cap D}$. It follows that $\overline{L}\subseteq \overline{U\cap D}$. Thus the countable set $D'$ is dense in the topological semilattice $Y$, which contradicts Theorem~\ref{BBK}.  
\end{proof}

\begin{proposition}\label{w1}
Let $X$ be a locally compact Nyikos topological semilattice. Then $X$ contains no chain isomorphic to $(\w_1,\min)$.      
\end{proposition}

\begin{proof}
To derive a contradiction assume that $X$ contains a chain $L=\{l_{\alpha}:\alpha<\w_1\}$ such that $l_{\alpha}\leq l_{\beta}$ if and only if $\alpha\leq \beta$. By \cite[Corollary 3.10.15]{Eng}, $X{\times}X$ is countable compact and thus pseudocompact. Theorem~\ref{BDG} implies that $\beta X$ is a compact topological semigroup. Since $X$ is dense in $\beta X$, we obtain that $\beta X$ is a semilattice. By Lemma~\ref{conv}, there exists $z=\sup L\in\beta X$.  
Lemma~\ref{good} implies that there exists $l_{\alpha}\in L$ such that for each neighborhood $U$ of $l_{\alpha}$ the set $U\setminus I_L$ contains an element $d_U$.
Since $X$ is locally compact, there exists a neighborhood $W$ of $l_{\alpha}$ such that $\cl_X(W)$ is compact.
 The continuity of the semilattice operation in $\beta X$ yields neighborhoods $U$ of $l_{\alpha}$ and $V$ of $z$ such that $UV\subseteq W$.
 By Lemma~\ref{conv}, there exists an ordinal $\delta<\w_1$ such that 
 $L'=\{l_{\xi}:\delta<\xi <\w_1\}\subseteq V$. Note that $d_UL'\subseteq W$, $L'$ is isomorphic to $(\w_1,\min)$ and $|d_UL'|=|d_UL|=\w_1$, by the choice of $d_U$. Since for each $a\in X$ the map $\phi_a: X\rightarrow X$, $\phi_a(x)=ax$ is a homomorphism, Lemma~\ref{folknew} implies that the chain $d_UL'$ is isomorphic to $(\beta,\min)$ for some ordinal $\beta\leq \w_1$. Taking into account that $|d_UL'|=\w_1$, we get that $d_UL'$ is isomorphic to $(\w_1,\min)$. Since $\cl_X(W)$ is compact, the chain $\cl_X(d_UL')\subseteq \cl_X(W)$ is compact, which contradicts Lemma~\ref{base}. 
\end{proof}

\begin{proposition}\label{w2}
Let $X$ be a locally compact Nyikos topological semilattice. Then $X$ contains no chain isomorphic to $(\w_1,\max)$.      
\end{proposition}

\begin{proof}
Seeking a contradiction, assume that a locally compact Nyikos semilattice $X$ contains a chain $L=\{l_{\alpha}:\alpha<\w_1\}$ such that $l_{\alpha}\leq l_{\beta}$ if and only if $\alpha\geq \beta$. By Theorem~\ref{BDG}, $\beta X$ is a compact topological semigroup. Since $X$ is dense in $\beta X$, we obtain that $\beta X$ is a semilattice. By Lemma~\ref{conv1}, there exists $z=\inf L\in \beta X\setminus X$.
Lemma~\ref{0} implies that the set $I_L=\{x\in X: |xL|<\w_1\}$ is not empty, as it contains the minimum of $X$. 

\begin{claim}\label{clopen}
The set $I_L$ is clopen in $X$.    
\end{claim}

\begin{proof}
By Lemma~\ref{IL1}, $I_L$ is closed. To derive a contradiction, assume that $I_L$ is not open. Then there exists $y\in I_L$ such that each neighborhood $U$ of $y$ contains an element $d_U\in X\setminus I_L$.   
Let us show that $yz\in X$. Since $y\in I_L$, we have $|yL|<\w_1$. Then there exists an uncountable and thus cofinal subset $A\subseteq \w_1$ such that $yl_\xi=yl_\delta$ for all $\xi,\delta\in A$. Let $\alpha=\min A$. It is routine to check that $yl_{\xi}=yl_{\alpha}$ for each $\xi\geq \alpha$. By Lemma~\ref{conv1}, for each neighborhood $V$ of $z$ there exists $\delta_V<\w_1$ such that $\{l_\xi: \xi\geq \delta_V\}\subseteq V$. Then $yl_{\alpha}\in yV$ for each neighborhood $V\subseteq \beta X$ of $z$. Since $\beta X$ is a Hausdorff topological semilattice, we get $yz=yl_{\alpha}\in X$. Since $X$ is locally compact, there is a neighborhood $W\subseteq X$ of $yz$ such that $\cl_X(W)$ is compact. The continuity of the semilattice operation in $\beta X$ yields neighborhoods $U$ of $y$ and $V$ of $z$ such that $UV\subseteq W$. 
Let $L'=\{l_{\alpha}:\alpha\geq \delta_V\}\subseteq V$. Note that $d_UL'\subseteq W$. By the choice of $d_U$, $|d_UL|=\w_1$, which implies $|d_UL'|=\w_1$. As left shifts in $\beta X$ are homomorphisms and $L'$ is order isomorphic to $(\w_1,\max)$, Lemma~\ref{folknew}(2) implies that $d_U L'$ is also order isomorphic to $(\w_1,\max)$. 
Since $\cl_X(W)$ is compact, the chain $\cl_X(d_UL')\subseteq \cl_X(W)$ is compact, which contradicts Lemma~\ref{base}.
\end{proof}

By~\cite[Corollary 3.6.5]{Eng}, the set $\cl_{\beta X}(I_L)$ is clopen.

\begin{claim}\label{cool}
There exists a finite set $F\subseteq X\setminus I_L$ such that $\beta X\setminus\cl_{\beta X}(I_L)=\bigcup_{x\in F}{\uparrow}x$.
\end{claim}

\begin{proof}
Let us first check that $\cl_{\beta X}(I_L)$ is an ideal in $\beta X$. It suffices to show that $ab\in \cl_{\beta X}(I_L)$ for any $a\in \cl_{\beta X}(I_L)$ and $b\in\beta X$. Fix a neighborhood $W$ of $ab$. By the continuity of the semigroup operation in $\beta X$, there exist neighborhoods $U$ of $a$ and $V$ of $b$ such that $UV\subseteq W$. Then there exist $c\in U\cap I_L$ and $d\in X\cap V$ such that $cd\in W\cap I_L$. Thus $ab\in \cl_{\beta X}(I_L)$, as required. 
Consider the Rees quotient semilattice $S=\beta X/{\cl_{\beta X}(I_L)}$ which is obtained by contracting the clopen ideal $\cl_{\beta X}(I_L)$ to a point denoted by $0$. Obviously, $S$ is a compact topological semilattice, $0=\inf S$ and $0$ is isolated in $S$. Further we agree to identify the sets $S\setminus\{0\}$ and $\beta X\setminus \cl_{\beta X}(I_L)$. Observe that the set $\beta X\setminus\cl_{\beta X}(I_L)$ is nonempty, as it contains $L$. Let $\mathfrak C$ be the set of all maximal chains in $S\setminus\{0\}$. By the compactness of $S\setminus\{0\}$, each $C\in\mathfrak C$ contains its minimum. To derive a contradiction, assume that the set 
$$F=\{\min C: C\in\mathfrak C\}\subseteq \beta X\setminus\cl_{\beta X}(I_L)$$ is infinite. It is easy to see that $ab=0$ for any distinct $a,b\in F$. So, $T=F\cup\{0\}$ is a subsemilattice of $S$. Since each chain in $T$ is finite, Theorem~\ref{H-closed} implies that $T$ is closed in $S$ and thus compact. Consider an accumulation point $y\in T$ of $F$. Observe that for each neighborhood $U$ of $y$ there exist distinct $a,b\in U\cap F$ witnessing that $0=ab\in UU$. Since $S$ is a Hausdorff topological semilattice, we get that $y=yy=0$. On the other hand, $y\neq 0$, as $0$ is an isolated point, a contradiction. So the set $F$ is finite.
At this point it is straightforward to check that $S\setminus\{0\}=\beta X\setminus\cl_{\beta X}(I_L)=\bigcup_{x\in F}{\uparrow}x$. 

It remains to check that $F\subseteq X$. Fix any $f\in F$. 
Observe that 
$$P:=S\setminus \big(\{0\}\cup\bigcup_{g\in F\setminus\{f\}}{\uparrow}g\big)=\{x\in{\uparrow}f: x\notin \bigcup_{g\in F\setminus\{f\}}{\uparrow}g\}$$ 
is an open subsemilattice of $S$ such that $\inf P=f\in P$. Then $P\cap X$ is an open subsemilattice of $X$. Taking into account that $X$ is separable, Lemma~\ref{0} yields that the semilattice $\cl_X(P\cap X)$ contains the smallest element $p$. Since $\cl_X(P\cap X)\subseteq \cl_{\beta X}(P)\subseteq {\uparrow}f$, we get $f\leq p$. Since $X$ is dense in $\beta X$ and $P$ is open, for every neighborhood $U$ of $f$ the set $U\cap P\cap X$ is not empty. Thus $p\in Up$ for each neighborhood $U$ of $f$. Since $S$ is a Hausdorff topological semilattice, $fp=p$, i.e. $p\leq f$. Hence $p=f$ and $f\in \cl_X(P\cap X)\subseteq X$.
\end{proof}

By Claim~\ref{clopen}, the sets $I_L$ and $X\setminus I_L$ are clopen. Corollary 3.6.2 from \cite{Eng} implies that $$\cl_{\beta X}(X\setminus I_L)\cap \cl_{\beta X}(I_L)=\varnothing.$$
Since $L\subseteq X\setminus I_L$, we obtain $$z\in \cl_{\beta X}(L)\subseteq  \cl_{\beta X}(X\setminus I_L)\subseteq \beta X\setminus \cl_{\beta X}(I_L).$$    
By Claim~\ref{cool}, there exists $f\in F\subseteq X\setminus I_L$ such that $f\leq z$. But then $fL=(fz)L=f(zL)=\{fz\}=\{f\}$. Hence $f\in I_L$, a contradiction. 
\end{proof}

\begin{theorem}\label{Nyikossemilattice}
A locally compact Nyikos topological semilattice is compact.    
\end{theorem}

\begin{proof}
Propositions~\ref{w1} and~\ref{w2} imply that each locally compact Nyikos topological semilattice $X$ contains no isomorphic copies of $(\w_1,\min)$ or $(\omega_1,\max)$. By Theorem~\ref{semilattice}, $X$ is compact.   
\end{proof}

\section{Proof of the main result and final remarks}\label{proof of the main result}

A partially ordered space $(X,\leq)$ endowed with a topology is called a {\em pospace} if $\leq$ is a closed subset of $X{\times}X$.
Each inverse semigroup $S$ carries the natural partial order $\leq$ defined by $x\leq y$ if and only if $x=xx^{-1}y$. 
 The Green's relations $\mathcal L$, $\mathcal R$, $\mathcal H$ and $\mathcal D$ on an inverse semigroup $S$ are defined as follows:

\begin{enumerate}[label=(\roman*)]
    \item $(x,y)\in\mathcal L$ if and only if $x^{-1}x=y^{-1}y$;
    \item $(x,y)\in\mathcal R$ if and only if $xx^{-1}=yy^{-1}$;
    \item $\mathcal H=\mathcal L\cap\mathcal R$;
    \item $\mathcal D=\mathcal L\circ\mathcal R=\mathcal R\circ\mathcal L$.
\end{enumerate}
For more about Green's relations on inverse semigroups see~\cite[Chapter 3.2]{Lawson}.
The following two results can be considered folklore.

\begin{proposition}[Folklore]\label{minposet}
Every nonempty compact subset of a pospace contains a minimal element.    
\end{proposition}

\begin{proposition}[Folklore]\label{Dclass}
 If $S$ is a topological inverse semigroup, then each two $\mathcal H$-classes within the one $\mathcal D$-class are homeomorphic.   
\end{proposition}

In fact, Proposition~\ref{Dclass} follows from the proof of \cite[Lemma~2.2.3]{Howie}. We are in a position to prove Main Theorem~\ref{main}, which states that every locally compact Nyikos inverse topological semigroup is compact. 

\begin{proof}[{\bf Proof of Main Theorem~\ref{main}}]
Consider a locally compact Nyikos inverse topological semigroup $S$. By Corollary~\ref{corNyikos}, $S$ is a topological inverse semigroup. Note that the semilattice of idempotents $E(S)$ is a retract of $S$ under the continuous map $\pi: x\mapsto xx^{-1}$. It follows that $E(S)$ is a locally compact Nyikos topological semilattice. Theorem~\ref{Nyikossemilattice} implies that $E(S)$ is compact. By Theorem~\ref{inverse}, $\beta S$ is a topological inverse semigroup. We claim that $E(S)=E(\beta S)$. Since $E(S)$ is compact, it suffices to show that $E(S)$ is dense in $E(\beta S)$. For this fix any $e\in E(\beta S)$ and its neighborhood $U$. Since $ee^{-1}=e$ and $\beta S$ is a topological inverse semigroup, there exists a neighborhood $V$ of $e$ such that $VV^{-1}\subseteq U$. Since $S$ is dense in $\beta S$, there exists $s\in S\cap V$. It follows that $ss^{-1}\in U\cap E(S)$. Since the neighborhood $U$ is chosen arbitrarily, we get that $e\in\cl_{\beta S}(E(S))$. Hence $E(S)$ is dense in $E(\beta S)$, which implies $E(S)=E(\beta S)$, as claimed.

Seeking a contradiction, assume that $\beta S\setminus S\neq \varnothing$. Since $S$ is locally compact, the remainder $\beta S\setminus S$ is closed and, consequently, compact. As $\beta S$ is a topological inverse semigroup, \cite[Proposition~3.8]{MPU} implies that $\beta S$ is a pospace with respect to the natural partial order. By Proposition~\ref{minposet}, $\beta S\setminus S$ contains a minimal element $h$. Let $e=hh^{-1}$ and $f=h^{-1}h$. Since $E(S)=E(\beta S)$, we have $e,f\in S$.  
Let $$T=\{x\in S: xx^{-1}\leq e \hbox{ and } x^{-1}x\leq f\}.$$ Fix a neighborhood $U$ of $h$. Since $ehf=h$, the continuity of the semigroup operation in $\beta S$ yields a neighborhood $V$ of $h$ such that $eVf\subseteq U$. Observe that for every $x\in S\cap V$ we have $exf\in T\cap U$. Since $U$ was chosen arbitrarily, $h\in \cl_{\beta S}(T)$.

Since $S$ is a topological inverse semigroup, the subgroups $H_e=\{x\in S\colon xx^{-1}=e=x^{-1}x\}$ and $H_f=\{x\in S\colon xx^{-1}=f=x^{-1}x\}$ are closed in $S$. Theorem~\ref{MU} implies that each closed subgroup of $S$ is compact. Thus, the sets $H_e$ and $H_f$ are compact. Observe that $H_e$ is the $\mathcal H$-class of the element $e$ and $H_f$ is the $\mathcal H$-class of the element $f$. 
Consider the $\mathcal H$-class $H_{e,f}:=\{x\in S\colon xx^{-1}=e \hbox{ and }x^{-1}x=f\}$. If $H_{e,f}\neq \varnothing$, then \cite[Proposition 5]{Lawson} implies that $H_{e,f}$ lies in the same $\mathcal D$-class with $H_e$ and $H_f$.  By Proposition~\ref{Dclass}, the set  $H_{e,f}$ is compact. Thus $h\in \cl_{\beta S}(T\setminus H_{e,f})$. Two cases are possible:
\begin{enumerate}[label=(\roman*)]
    \item $h\in \cl_{\beta S}(\{x\in S: xx^{-1}<e\hbox{ and }x^{-1}x\leq f\})$;
       \item $h\in \cl_{\beta S}(\{x\in S: xx^{-1}\leq e\hbox{ and }x^{-1}x< f\})$. 
\end{enumerate}

(i) Since $e=hh^{-1}$ and $\beta S$ is a topological inverse semigroup, for each neighborhood $U$ of $e$ there exists a neighborhood $V$ of $h$ such that $VV^{-1}\subseteq U$. By assumption, there is $x\in V$ such that $xx^{-1}<e$ and $xx^{-1}\in U$. It follows that $e\in \cl_S({\downarrow}e\setminus\{e\})$. Since $S$ is first-countable, there exists a sequence $\{e_n:n\in\w\}\subseteq {\downarrow}e\setminus\{e\}$ that converges to $e$. Since $eh=h$ we get that the sequence $\{e_nh:n\in\w\}$ converges to $h$. Note that $e_nh\neq h$ for each $n\in\w$, as otherwise $e=hh^{-1}=e_nhh^{-1}e_n=e_nee_n=e_n$ for some $n\in\w$, which contradicts the choice of $e_n$. Since $e_nhh^{-1}e_nh=e_nh$, we get that $e_nh\leq h$ with respect to the natural partial order on $\beta S$. Since $h$ is a minimal element in $\beta S\setminus S$, we get that $e_nh\in S$ for all $n\in\w$. It follows that $\{e_nh:n\in\w\}$ is an infinite closed discrete subset of $S$, which contradicts the countable compactness of $S$.

(ii) Since $f=h^{-1}h$ and $\beta S$ is a topological inverse semigroup, for each neighborhood $U$ of $f$ there exists a neighborhood $V$ of $h$ such that $V^{-1}V\subseteq U$. By assumption, there is $x\in V$ such that $f>x^{-1}x\in U$. It follows that $f\in \cl_S({\downarrow}f\setminus\{f\})$. Since $X$ is first-countable, there exists a sequence $\{f_n:n\in\w\}\subseteq {\downarrow}f\setminus\{f\}$ that converges to $f$. Since $hf=h$ we get that the sequence $\{hf_n:n\in\w\}$ converges to $h$. Similarly as above it can be checked that $h\notin\{hf_n:n\in\w\}$ and $hf_n\leq h$ for every $n\in\w$. Since $h$ is a minimal element in $\beta S\setminus S$, we get that $hf_n\in S$ for all $n\in\w$. It follows that $\{hf_n:n\in\w\}$ is an infinite  closed discrete subset of $S$, which contradicts the countable compactness of $S$. 

The obtained contradictions imply $S=\beta S$. 
\end{proof}

A {\em band} is a semigroup consisting of idempotents. A semigroup $S$ is called {\em simple} if $S$ contains no proper two-sided ideals or, in other words, $SxS=S$ for any $x\in S$. 
The following example shows that Main Theorem~\ref{main} does not generalize over simple bands.

\begin{example}
Let $X$ be one of the consistent examples of a locally compact noncompact Nyikos space discussed in the introduction. Let $X_1$ be the space $X$ endowed with the left zero operation, i.e., $xy=x$ for each $x,y\in X$. It is easy to see that this operation is continuous and associative, that is, $X_1$ is a topological semigroup. Let $X_2$ be the space $X$ endowed with the right zero operation, i.e., $xy=y$ for each $x,y\in X$. Similarly, one can check that $X_2$ is a topological semigroup.
Taking into account \cite[Corollary 3.10.15]{Eng}, it is easy to check that the finite product of Nyikos spaces remains Nyikos. Hence the direct product $S:=X_1{\times}X_2$ is a noncompact Nyikos topological semigroup. Note that $(a,b)(a,b)=(a,b)$ for every $(a,b)\in S$, which makes $S$ a band. Fix $(a,b)\in S$. Then for any $(c,d)\in S$ we have $(c,d)=(c,a)(a,b)(b,d)\in S(a,b)S$, witnessing that $S=S(a,b)S$. Since the point $(a,b)$ was chosen arbitrarily, the band $S$ is simple.  
\end{example}

The {\em Ostaszewski space} was constructed under ($\diamondsuit$) in~\cite{Ost}. It is a Nyikos space and possesses (among others) the following properties: locally compact, hereditary separable, each compact subspace is countable. 
The results of this paper establish another peculiar property of the Ostaszewski space.

\begin{proposition}
Let $S$ be the Ostaszewski space endowed with a continuous semigroup operation. Then each inverse subsemigroup of $S$ is countable.    
\end{proposition}

\begin{proof}
Let $Y$ be an inverse subsemigroup of $S$. Note that $\overline{Y}$ is a locally compact first-countable countably compact topological subsemigroup of $S$. Since $S$ is hereditary separable, $\overline{Y}$ is a locally compact Nyikos topological semigroup. Corollary~\ref{corNyikos} implies that $\overline{Y}$ is a topological inverse semigroup. By Main Theorem~\ref{main}, $\overline{Y}$ is compact. Since each compact subspace of the Ostaszewski space is countable, we get $|\overline{Y}|\leq\w$. Hence $|Y|\leq \w$, as required.    
\end{proof}

We finish this paper with the following conjecture.

\begin{conjecture}
Each Nyikos inverse topological semigroup is compact.    
\end{conjecture}


\subsection*{Acknowledgements}
The research of the author was funded in whole by the Austrian Science Fund FWF [10.55776/ESP399].

\end{document}